\documentclass[english,letterpaper,11pt,reqno]{amsart}
\usepackage[backend=bibtex,style=alphabetic,maxnames=6]{biblatex}
\usepackage{appendix}
\usepackage{amsbsy}
\usepackage{scalerel}
\usepackage{tcolorbox}
\usepackage{amsfonts}
\usepackage{amsmath}
\usepackage{amssymb}
\usepackage{amsthm}
\usepackage{graphicx}
\usepackage{ifthen}
\usepackage{textcomp}
\usepackage{enumitem, color, amssymb}
\usepackage{dsfont}
\usepackage{mathrsfs}
\usepackage{calrsfs}
\usepackage[bookmarksnumbered,colorlinks]{hyperref}
\hypersetup{colorlinks=true, linkcolor=blue, citecolor=blue}
\usepackage{cancel}
\usepackage{setspace}

\newcommand{\lra}{\longrightarrow}

\newcommand{\RR}{\mathbb{R}}

\newcommand{\vepo}{\epsilon_{\scalebox{0.4}{\emph{L}}}}

\makeatletter
\newcommand*{\defeq}{\mathrel{\rlap{%
                     \raisebox{0.25ex}{$\m@th\cdot$}}%
                     \raisebox{-0.25ex}{$\m@th\cdot$}}%
                     =}
\makeatother

\makeatletter
\newcommand*\owedge{\mathpalette\@owedge\relax}
\newcommand*\@owedge[1]{%
  \mathbin{%
    \ooalign{%
      $#1\m@th\bigcirc$\cr
      \hidewidth$#1\m@th\wedge$\hidewidth\cr
    }%
  }%
}
\makeatother

\newtheorem{thm}{Theorem}

\newtheorem{cor}{Corollary}

\newtheorem{defn}{Definition}
\newtheorem{prop}{Proposition}
\newtheorem*{definition-non}{Definition}
\newtheorem*{theorem-non}{Theorem}
\newtheorem*{proposition-non}{Proposition}
\newtheorem*{lemma-non}{Lemma}
\newtheorem*{corollary-non}{Corollary}

\newcommand{\beqa}{\begin{eqnarray}}
\newcommand{\beq}{\begin{equation}}
\newcommand{\eeqa}{\end{eqnarray}}
\newcommand{\eeq}{\end{equation}}

\newcommand\ipl[2]{\langle {#1},{#2}\rangle_{\!g_{\scalebox{0.3}{\emph{L}}}}}
\newcommand\ipr[2]{\langle {#1},{#2}\rangle_{\!\scalebox{0.7}{\emph{g}}}}

\newcommand\ww[2]{#1 \wedge #2}

\newcommand\cd[2]{\nabla_{\!#1}{#2}}

\newcommand\gL{g_{\scalebox{0.4}{$L$}}}
\newcommand\ggL{{\bf g}_{\scalebox{0.4}{$L$}}}

\newcommand\Ric{\text{Ric}}

\newcommand\Rmr{\text{Rm}}

\newcommand\WL{W_{\scalebox{0.4}{$L$}}}

\newcommand\comma{\hspace{.2in},\hspace{.2in}}
\newcommand\commas{\hspace{.1in},\hspace{.1in}}
\newcommand\commass{\hspace{.05in},\hspace{.05in}}

\newcommand\hsl{*_{\scalebox{0.4}{$L$}}}

\newcommand\hsr{*}

\newcommand\cow{\hat{W}}
\newcommand\coww{\hat{W}^{\scalebox{0.55}{$\gL$}}}
\newcommand\cowT{\hat{W}_{\scalebox{0.4}{$L$}}^{\scalebox{0.5}{$g$}}}
\newcommand\cx{\Lambda_{\scalebox{0.5}{\emph{$\mathbb{C}$}}}^2}

\newcommand\co{\hat{R}}

\newcommand\Wsec{\text{sec}_{\scalebox{0.5}{\emph{\emph{$\coww$}}}}}
\newcommand\WTsec{\text{sec}_{\scalebox{0.5}{\emph{\emph{$\cowT$}}}}}

\newtheorem{innercustomgeneric}{\customgenericname} 
\providecommand{\customgenericname}{}
\newcommand{\newcustomtheorem}[2]{%
  \newenvironment{#1}[1]
  {%
   \renewcommand\customgenericname{#2}%
   \renewcommand\theinnercustomgeneric{##1}%
   \innercustomgeneric
  }
  {\endinnercustomgeneric}
}

\newcustomtheorem{customthm}{Theorem}
\newcustomtheorem{customlemma}{Lemma}

\begin{document}
\title[]{Petrov Types for the Weyl tensor via the Riemannian-to-Lorentzian Bridge}
\author[]{Amir Babak Aazami}
\address{Clark University\hfill\break\indent
Worcester, MA 01610}
\email{aaazami@clarku.edu}

\maketitle
\begin{abstract}
We analyze oriented Riemannian 4-manifolds whose Weyl tensors $W$ satisfy the conformally invariant condition $W(T,\cdot,\cdot,T) = 0$ for some nonzero vector $T$.  While this can be algebraically classified via $W$'s normal form, we find a further geometric classification by deforming the metric into a Lorentzian one via $T$. We show that such a $W$ will have the analogue of Petrov Types from general relativity, that only Types I and D can occur, and that each is completely determined by the number of critical points of $W$'s associated Lorentzian quadratic form. A similar result holds for the Lorentzian version of this question, with $T$ timelike.
\end{abstract}

\section{Introduction}
The Weyl curvature tensor $W$ of an oriented Riemannian 4-manifold, being trace-free, is well known to possess a \emph{normal form}: Relative to any orthonormal basis, the linear endomorphism of $W$ on $\Lambda^2$ always has the block form ${\tiny \begin{bmatrix}
A & B\\
B & A
\end{bmatrix}}$. The $3 \times 3$ matrices $A,B$ are symmetric, and by the classical work of \cite{berger,thorpe2}, there exist bases relative to which $A,B$ are diagonal and completely determined by just the critical points and values of $W$'s associated quadratic form on $\Lambda^2$. (If the metric was Einstein, then this would be true with the full curvature 4-tensor in place of $W$, and the quadratic form would be the sectional curvature.)  Given this normal form for $W$, it is thus natural to investigate additional forms of symmetry. One may, e.g., analyze the case when $A = \pm B$, and this is well known to occur if and only if the 4-manifold is \emph{self-dual} or \emph{anti-self-dual}. There are many such 4-manifolds (see, e.g., \cite{lebrun2,taubes}), and anti-self-dual 4-manifolds in particular are related to R.~Penrose's twistor program \cite{besse,lebrun}.
\vskip 6pt
Here we analyze another condition on $W$'s normal form, namely, the condition that $A=O$ in at least one orthonormal basis. Denoting that ordered basis by $\{e_1,e_2,e_3,e_4\}$, that $A=O$ in this basis is equivalent to $W(e_1,\cdot,\cdot,e_1) = 0$, which is conformally invariant.  After providing some examples and non-examples of such 4-manifolds, we proceed to classifying them. To begin with, in any basis $\{e_i\}$, if some $T = \sum_i c_ie_i$ satisfies $W(T,\cdot,\cdot,T) = 0$, then the equations $W(T,e_j,e_k,T) = 0$ yield an overdetermined system for the $c_i$'s. If one takes $\{e_i\}$ to be a basis that diagonalizes $W$ \`a la \cite{berger,thorpe2}, then these equations become more tractable (Theorem \ref{thm:alg}), and certain consequences can be drawn. However, to approach this problem in a more geometric, basis-independent manner, it is best to first change the metric. Namely, if $W(T,\cdot,\cdot,T) = 0$ for some metric $g$ and unit-length vector field $T$, then $W$ will also be trace-free with respect to the \emph{Lorentzian} metric $\gL \defeq g - 2T^\flat \otimes T^\flat$. If one then considers $W$'s $\gL$-induced endomorphism on $\Lambda^2$, then it will have $\gL$-normal forms, analogous to the Petrov Types of general relativity \cite{Petrov,o1995,stephani}. Indeed, just as \cite{thorpe} showed how the five possible Petrov Types are determined by the number of critical points of the spacetime's sectional curvature function, in Theorems \ref{thm:n2} and \ref{thm:onlyI} we prove something similar: $W$ can have only two $\gL$-normal forms, and they are determined by the number of critical points of its associated $\gL$-quadratic form:
\begin{theorem-non}
Let $(M,g)$ be an oriented Riemannian 4-manifold whose Weyl tensor $W$ satisfies $W(T,\cdot,\cdot,T) = 0$ for some unit-length vector field $T$ on $M$. Set $\gL \defeq g - 2T^{\flat}\otimes T^{\flat}$ and let $\coww\colon\Lambda^2\lra \Lambda^2$ denote the linear endomorphism of $W$ defined via the $\gL$-induced Lorentzian inner product \emph{$\ipl{\,}{}$} on $\Lambda^2$\emph{:}
$$
\text{\emph{$\ipl{\coww(v\wedge w)}{x\wedge y} \defeq W(v,w,x,y)$}}\hspace{.2in}\text{for all $v,w,x,y \in T_pM$}.
$$
Then at each $p \in M$, the $\gL$-quadratic form of $W$ given by \eqref{def:T_seccoo1} always has either $3$ or $\infty$ critical points, which number completely determines its two possible $\gL$-normal forms \emph{(}Petrov Types\emph{)} as defined in Definition \ref{def:Ptype}.
\end{theorem-non}

In other words, there is basis-independent geometry in the condition $A=O$ that began our inquiry\,---\,though we point out that, unlike in the Riemannian case, the critical points and values here do not determine the invariants of $W$'s $\gL$-normal form (i.e., the analogues of the entries in $A,B$); for that, one would also need the Hessian of \eqref{def:T_seccoo1}. Also, $T$ needn't be unique here, either. Nevertheless, Theorems \ref{thm:n2} and \ref{thm:onlyI} show the continued presence of normal forms in dimension four, a presence that is already rich: From their use in the proof of the Hitchin-Thorpe inequality \cite{thorpe3,hitchin}), to their relation to the Petrov Types of general relativity, and more recently to extensions of the curvature tensor, e.g., in gradient Ricci solitons \cite{cao}.

\vskip 6pt
Finally, we reverse the roles of $g$ and $\gL$, by considering when the Weyl tensor $\WL$ of an oriented Lorentzian 4-manifold $(M,\gL)$ satisfies $\WL(T,\cdot,\cdot,T) = 0$ for some unit-length \emph{timelike} $T$. By passing over to the Riemannian metric $g \defeq \gL+2T^{\flat}\otimes T^{\flat}$, we use \cite{berger,thorpe2} to prove in Theorem \ref{thm:n2R} the following:
\begin{theorem-non}
Let $(M,\gL)$ be an oriented Lorentzian 4-manifold whose Weyl tensor $\WL$ satisfies
$\WL(T,\cdot,\cdot,T) = 0$
for some unit-length timelike vector field $T$ on $M$. Then $(M,\gL)$ cannot have Petrov Types II, N, or III. Set $g \defeq \gL + 2T^{\flat}\otimes T^{\flat}$ and let $\cowT\colon\Lambda^2\lra \Lambda^2$ denote the linear endomorphism of $\WL$ defined via the $g$-induced Riemannian inner product \emph{$\ipr{\,}{}$} on $\Lambda^2$\emph{:}
$$
\text{\emph{$\ipr{\cowT(v\wedge w)}{x\wedge y} \defeq \WL(v,w,x,y)$}}\hspace{.2in}\text{for all $v,w,x,y \in T_pM$}.
$$
Then at each $p\in M$, $\WL$ is completely determined by the critical points and values of $\WL$'s $g$-quadratic form \emph{$P \mapsto \ipr{\cowT P}{P}$}, defined for all 2-planes $P$.
\end{theorem-non}

In contrast to the Petrov Types of general relativity, whose invariants cannot be determined by just the first derivatives of their quadratic forms (see \cite{thorpe}; their Hessian is also needed), Theorem \ref{thm:n2R} shows that those Weyl tensors that satisfy $\WL(T,\cdot,\cdot,T) = 0$ \emph{can} be determined by just the first derivatives\,---\,of a $g$-induced quadratic form.  Finally, we remark that the proofs of Theorems \ref{thm:n2} and \ref{thm:n2R} are completely classical; indeed, they follow those of \cite{thorpe,thorpe2} step-by-step\,---\,but with one important difference: The Weyl endomorphism originates from $g$, whereas the Hodge star operator with which it commutes originates from $\gL$, and vice versa. As such, this article follows an ongoing line of inquiry begun in \cite{aazami}, namely, that of studying the commutativity of curvature operators and Hodge stars arising from \emph{different} metrics over the same 4-manifold $M$.

\section{Brief review of the Weyl tensor}
\label{sec:review}
In this section we briefly review the Weyl tensor and establish our notation. Let $(M,g)$ be an oriented Riemannian 4-manifold, with Riemann curvature 4-tensor $\Rmr(v,w,x,y) = g(\cd{v}{\cd{w}{x}}-\cd{w}{\cd{v}{x}}-\cd{[v,w]}{x},y)$. Recall that $g$'s \emph{curvature operator} $\co\colon \Lambda^2 \lra \Lambda^2$ is the linear endomorphism defined by  
\beqa
\label{eqn:co0}
\ipr{\co(v\wedge w)}{x\wedge y} \defeq \Rmr(v,w,x,y)\hspace{.2in}\text{for all $v,w,x,y \in T_pM$},
\eeqa
where $\ipr{\,}{}$ is the $g$-induced inner product $\ipr{\,}{}$ on $\Lambda^2$:
\beqa
\label{eqn:gonm}
\ipr{\ww{v}{w}}{\ww{x}{y}}  \defeq \text{det}\begin{bmatrix}
g(v,x) & g(v,y)\\
g(w,x) & g(w,y)
\end{bmatrix}\cdot
\eeqa
(If \eqref{eqn:co0} is defined with a minus sign, then $\ipr{\co(v\wedge w)}{v\wedge w} = \Rmr(v,w,w,v)$ would be the sectional curvature of the 2-plane spanned by the orthonormal pair $v,w$; \cite{thorpe,thorpe2} define it as we have.) Note that the pairwise symmetry $\Rmr_{ijkl}=\Rmr_{klij}$ ensures that $\co$ is $\ipr{\,}{}$-self-adjoint. Now, any orthonormal basis $\{e_1,e_2,e_3,e_4\} \subseteq T_pM$ lifts to a $\ipr{\,}{}$-orthonormal basis
\beqa
\label{eqn:Hbasis0}
\{\ww{e_1}{e_2}\,,\,\ww{e_1}{e_3}\,,\,\ww{e_1}{e_4}\,,\, \ww{e_3}{e_4}\,,\, \ww{e_4}{e_2}\,,\, \ww{e_2}{e_3}\} \subseteq \Lambda^2.
\eeqa
Relative to this basis, and writing $\Rmr(e_i,e_j,e_k,e_l) \defeq R_{ijkl}$, the curvature operator \eqref{eqn:co0} takes the block form
\beqa
\label{eqn:co1}
\co = \begin{bmatrix}
R_{1212} & R_{1312} & R_{1412} & R_{3412} & R_{4212} & R_{2312}\\
R_{1213} & R_{1313} & R_{1413} & R_{3413} & R_{4213} & R_{2313}\\
R_{1214} & R_{1314} & R_{1414} & R_{3414} & R_{4214} & R_{2314}\\
R_{1234} & R_{1334} & R_{1434} & R_{3434} & R_{4234} & R_{2334}\\
R_{1242} & R_{1342} & R_{1442} & R_{3442} & R_{4242} & R_{2342}\\
R_{1223} & R_{1323} & R_{1423} & R_{3423} & R_{4223} & R_{2323}
\end{bmatrix} = \begin{bmatrix}
A & B\\
B^t & D
\end{bmatrix},
\eeqa
with $A$ and $D$ symmetric $3 \times 3$ matrices and $B^t$ the transpose of the $3 \times3$ matrix $B$, which is not symmetric in general. The second linear operator we shall need is the \emph{Hodge star operator} $\hsr\colon \Lambda^2\lra \Lambda^2$, defined via
$$
\ww{\xi}{\hsr\eta} \defeq \ipr{\xi}{\eta}\,dV \comma \xi,\eta \in \Lambda^2,
$$
where $dV \in \Lambda^4$ is the orientation form. The action of $\hsr$ on the basis \eqref{eqn:Hbasis0} is given by
\beqa
\arraycolsep=1.4pt\def\arraystretch{1.5}
\left\{\begin{array}{lr}
\hsr(\ww{e_1}{e_2}) = \ww{e_3}{e_4},\\
\hsr(\ww{e_1}{e_3}) = \ww{e_4}{e_2},\\
\hsr(\ww{e_1}{e_4}) = \ww{e_2}{e_3},
\end{array}\right.  \comma \arraycolsep=1.4pt\def\arraystretch{1.5}
\left\{\begin{array}{lr}
\hsr(\ww{e_3}{e_4}) = \ww{e_1}{e_2},\\
\hsr(\ww{e_4}{e_2}) = \ww{e_1}{e_3},\\
\hsr(\ww{e_2}{e_3}) = \ww{e_1}{e_4},\label{eqn:*1}
\end{array}\right.
\eeqa
or in block matrix form,
\beqa
\label{eqn:HbasisR}
\hsr = \begin{bmatrix}
O & I\\
I & O
\end{bmatrix},
\eeqa
where $O$ and $I$ are the $3\times 3$ zero and identity matrices, respectively. Observe that $\hsr$ has the eigenvalues $\pm 1$, whose eigenbases $\Lambda^\pm$, which decompose $\Lambda^2$ into the direct sum $\Lambda^2=\Lambda^+\oplus \Lambda^-$, are given in terms of \eqref{eqn:Hbasis0} by
$$
\Lambda^{\pm} \defeq \text{span}\Big\{\frac{1}{\sqrt{2}}(\ww{e_1}{e_2} \pm \ww{e_3}{e_4}),\frac{1}{\sqrt{2}}(\ww{e_1}{e_3} \pm \ww{e_4}{e_2}),\frac{1}{\sqrt{2}}(\ww{e_1}{e_4} \pm \ww{e_2}{e_3})\Big\}\cdot
$$
Eigenvectors $\xi \in \Lambda^+$ are called \emph{self-dual} ($\hsr \xi = +\xi$), while $\xi \in \Lambda^-$ are \emph{anti-self-dual} ($\hsr \xi = -\xi$). We now briefly present the \emph{Weyl curvature tensor} of a Riemannian 4-manifold $(M,g)$. It is defined by
\beqa
\label{eqn:Weyl_form}
W \defeq \Rmr - \frac{1}{2}\Ric\,{\tiny \owedge}\,g + \frac{\text{scal}_{\scalebox{0.6}{\emph{g}}}}{12}g\,{\tiny \owedge}\,g,
\eeqa
where ${\tiny \owedge}$ is the Kulkarni-Nomizu product (see \cite[p.~213]{Lee}), $\Ric$ is the Ricci tensor of $g$, and $\text{scal}_{\scalebox{0.6}{\emph{g}}}$ its scalar curvature.  The most important facts about $W$ is that it is trace-free and that it is the conformally invariant part of the Riemann curvature 4-tensor. The latter means that for any metric $\tilde{g} \defeq e^{2f}g$ in the conformal class of $g$, its Weyl tensor $\widetilde{W}$ scales as $\widetilde{W} = e^{2f}W$, and furthermore that $W = 0$ is precisely the condition required for $g$ to be locally conformally flat. Just as with \eqref{eqn:co0}, $W$ also has a corresponding linear endomorphism $\hat{W}\colon\Lambda^2 \lra \Lambda^2$, defined by
\beqa
\label{eqn:coW}
\ipr{\hat{W}(v\wedge w)}{x\wedge y} \defeq W(v,w,x,y)\hspace{.2in}\text{for all $v,w,x,y \in T_pM$}.
\eeqa
Because $W$ is trace-free and otherwise satisfies the same symmetry properties as the Riemann curvature 4-tensor, relative to any $\ipr{\,}{}$-orthonormal basis of the form \eqref{eqn:Hbasis0} for $\Lambda^2$, $\cow$ has the block form
\beqa
\label{eqn:WW}
\cow = \begin{bmatrix}
W_{1212} & W_{1312} & W_{1412} & W_{3412} & W_{4212} & W_{2312}\\
W_{1213} & W_{1313} & W_{1413} & W_{3413} & W_{4213} & W_{2313}\\
W_{1214} & W_{1314} & W_{1414} & W_{3414} & W_{4214} & W_{2314}\\
W_{1234} & W_{1334} & W_{1434} & W_{3434} & W_{4234} & W_{2334}\\
W_{1242} & W_{1342} & W_{1442} & W_{3442} & W_{4242} & W_{2342}\\
W_{1223} & W_{1323} & W_{1423} & W_{3423} & W_{4223} & W_{2323}
\end{bmatrix} = \begin{bmatrix}
A & B\\
B & A
\end{bmatrix}\cdot
\eeqa
(I.e., $W_{1212} = W_{3434}$, etc.) A glance at \eqref{eqn:HbasisR} shows that, as a consequence of being trace-free, $\hat{W}$ commutes with $\hsr$, an important fact to which we shall return in Section \ref{sec:Riem} below.  In terms of $\co$ and $\hsr$, $\hat{W}$ can also be expressed as
\beqa
\label{eqn:Wnew}
\hat{W} = \frac{1}{2}(\co + \hsr \circ \co \circ \hsr) + \frac{\text{scal}_{\scalebox{0.6}{$g$}}}{12}I,
\eeqa
where $I\colon \Lambda^2\lra \Lambda^2$ is the identity operator.  Using \eqref{eqn:Wnew} and the fact that $\hsr$ is $\ipr{\,}{}$-self-adjoint, it follows easily that the eigenspaces $\Lambda^{\pm}$ are $\hat{W}$-invariant. Thus, along with $\Lambda^2 = \Lambda^+\oplus \Lambda^-$, there is a decomposition $\hat{W} = W^+ \oplus W^-$ into endomorphisms $W^{\pm}\colon\Lambda^{\pm}\lra \Lambda^{\pm}$, respectively. An oriented Riemannian 4-manifold is said to be \emph{self-dual} if $W^- = 0$ and \emph{anti-self-dual} if $W^+ = 0$. In terms of \eqref{eqn:WW}, $W^{\pm}=0 \!\!\iff\!\! A = \pm B$ (see \eqref{eqn:check1}-\eqref{eqn:check2} below), which, recalling $\widetilde{W} = e^{2f}W$, makes it clear that these are conformally invariant conditions. The final map that we will need is \emph{$W$'s associated quadratic form}. To define it, let $G_p \subseteq \Lambda^2$ denote the set of decomposable 2-vectors of $\ipr{\,}{}$-length 1; i.e., all 2-vectors $\xi \in \Lambda^2$ of the form $\xi = \ww{v}{w}$ for some $v,w \in T_pM$, with $\ipr{\xi}{\xi} = 1$; these are also called \emph{2-planes at $T_pM$}. These, too, can be characterized by $\hsr$ in dimension four: $G_p = \{\xi^++\xi^- : \xi^{\pm} \in \Lambda^{\pm}, \ipr{\xi^{\pm}}{\xi^{\pm}}=1/2\}$; see, e.g., \cite{thorpe2}. The quadratic form of $\hat{W}$ is then the  map defined by
\beqa
\label{eqn:qaudW}
P \mapsto \ipr{\hat{W}P}{P} \comma P \in G_p.
\eeqa
(Notice its relation to \eqref{eqn:coW}, exactly as sectional curvature related to the curvature operator \eqref{eqn:co0}.) This map, and its Lorentzian analogue in \cite{thorpe} (a variant of which will appear in Definition \ref{def:T_sec0coo1} below), will play a fundamental role in our study of the condition $W(T,\cdot,\cdot,T)=0$, to which we now turn.

\section{The condition $W(T,\cdot,\cdot,T) = 0$}
\label{sec:Riem}
We now proceed to our objects of study, oriented Riemannian 4-manifolds $(M,g)$ whose Weyl tensors $W$ satisfy the pointwise condition
\beqa
\label{eqn:main}
\text{$W(T,\cdot,\cdot,T) = 0 \comma$ some unit $T \in T_pM$}.
\eeqa
This is a conformally invariant condition, trivial $g$ is locally conformally flat. For an example that is not locally conformally flat (or Einstein), consider the metric $g$ defined on $M \defeq \{(r,x,y,z) \in \RR^4 : x > 0\}$ by
$$
g \defeq (2x)^3(dr)^2 + (dx)^2+(2x)^{-3}(dy)^2+(dz)^2.
$$
The components of the Weyl tensor of $g$, relative to the orthonormal basis $\{e_1,e_2,e_3,e_4\} \defeq \{\partial_r/\sqrt{8x^3},\partial_x,\sqrt{8x^3}\partial_y,\partial_z\}$, are given by

{\tiny \beqa
\cow = \begin{bmatrix}
W_{1212} & W_{1312} & W_{1412} & W_{3412} & W_{4212} & W_{2312}\\
W_{1213} & W_{1313} & W_{1413} & W_{3413} & W_{4213} & W_{2313}\\
W_{1214} & W_{1314} & W_{1414} & W_{3414} & W_{4214} & W_{2314}\\
W_{1234} & W_{1334} & W_{1434} & W_{3434} & W_{4234} & W_{2334}\\
W_{1242} & W_{1342} & W_{1442} & W_{3442} & W_{4242} & W_{2342}\\
W_{1223} & W_{1323} & W_{1423} & W_{3423} & W_{4223} & W_{2323}
\end{bmatrix} = \frac{3}{2x^2}\!\begin{bmatrix}
0 & 0 & 0 & 0 & 0 & 0\\
0 & -1 & 0 & 0 & 0 & 0\\
0 & 0 & 1 & 0 & 0 & 0\\
0 & 0 & 0 & 0 & 0 & 0\\
0 & 0 & 0 & 0 & -1 & 0\\
0 & 0 & 0 & 0 & 0 & 1
\end{bmatrix}\cdot\label{eqn:*}
\eeqa}

It can be verified that $T \defeq \frac{1}{\sqrt{2}}(e_1\pm e_2),\frac{1}{\sqrt{2}}(e_3\pm e_4)$ all satisfy $W(T,\cdot,\cdot,T) = 0$ on $M$. On the other hand, here are two classes of Riemannian 4-manifolds that are both ``close" to being locally conformally flat but which nevertheless do not satisfy \eqref{eqn:main}:
\begin{enumerate}[leftmargin=*]
\item[i.] \emph{Products of constant curvature surfaces}: Let $(M_1,g_1),(M_2,g_2)$ be two oriented Riemannian 2-manifolds with constant sectional curvatures $c_1,c_2$. Then their product $(M_1\times M_2,g_1\oplus g_2)$ is an oriented 4-manifold whose Weyl tensor is given by
$$
W = \frac{c_1+c_2}{6}\Big(g_1\,{\tiny \owedge}\,g_1 - g_1\,{\tiny \owedge}\,g_2 + g_2\,{\tiny \owedge}\,g_2\Big)
$$
(see \cite[p.~261]{Lee}). If $c_1 \neq -c_2$, then $W \neq 0$. However, for any choice of $T = ax_1+bx_2$ with $x_1 \in T_{p_1}M_1$ and $x_2 \in T_{p_2}M_2$ vectors of unit length and $a,b\in \RR$, 
\beqa
W(T,x_1,x_1,T) \!\!&=&\!\! b^2W(x_2,x_1,x_1,x_2) = \frac{c_1+c_2}{6}(-b^2),\nonumber\\
W(T,x_2,x_2,T) \!\!&=&\!\! a^2W(x_1,x_2,x_2,x_1) = \frac{c_1+c_2}{6}(-a^2),\nonumber
\eeqa
and these vanish if and only if $a=b=0$. Thus if $W \neq 0$, then \eqref{eqn:main} cannot be satisfied for any nonzero $T$.
\item[ii.] \emph{Self-dual/anti-self-dual 4-manifolds}: The easiest way to verify this is to recall that $W^{\pm} = 0 \iff A=\pm B$, together with the fact that \eqref{eqn:main} implies that $A=O$ in some orthonormal basis (see \eqref{eqn:cond2} below). But let us verify this directly here. Thus, suppose that \eqref{eqn:main} holds for a self-dual metric ($W^- = 0$), so that $\hat{W}(\xi) = 0$ for all $\xi \in \Lambda^-$. Then taking an oriented orthonormal basis of the form $\{T=e_1,e_2,e_3,e_4\}$, and forming the following three eigenvectors of $\Lambda^-$,
$$
\hspace{.25in}\Big\{\frac{1}{\sqrt{2}}(\ww{e_1}{e_2} - \ww{e_3}{e_4}),\frac{1}{\sqrt{2}}(\ww{e_1}{e_3} - \ww{e_4}{e_2}),\frac{1}{\sqrt{2}}(\ww{e_1}{e_4} - \ww{e_2}{e_3})\Big\},
$$
it will be the case that
$$
\hat{W}(\ww{e_1}{e_2}) = \hat{W}(\ww{e_3}{e_4}) \commass \hat{W}(\ww{e_1}{e_3}) = \hat{W}(\ww{e_4}{e_2}) \commass \hat{W}(\ww{e_1}{e_4}) = \hat{W}(\ww{e_2}{e_3}).  
$$
Consider the first of these; because $W_{1212} = W_{1213} = W_{1214} = 0$ by \eqref{eqn:main}, we have
\beqa
\label{eqn:check1}
\hat{W}(\ww{e_1}{e_2}) = W_{1234}\ww{e_3}{e_4}+W_{1242}\ww{e_4}{e_2}+W_{1223}\ww{e_2}{e_3}.
\eeqa
On the other hand, because $W_{3434} = W_{1212} = 0$, $W_{3442} = W_{1312} = 0$, and $W_{3423} = W_{1412} = 0$ (recall \eqref{eqn:WW}), we also have
\beqa
\label{eqn:check2}
\hat{W}(\ww{e_3}{e_4}) = W_{3412}\ww{e_1}{e_2}+W_{3413}\ww{e_1}{e_3}+W_{3414}\ww{e_1}{e_4}.
\eeqa
Thus the only way to satisfy $\hat{W}(\ww{e_1}{e_2}) = \hat{W}(\ww{e_3}{e_4})$ is for both to be zero. Similarly for the other two cases, as well as for the anti-self-dual case $(W^+=0)$. We thus conclude that any self-dual or anti-self-dual 4-manifold that is not locally conformally flat cannot satisfy \eqref{eqn:main}.
\end{enumerate}

To understand when $W(T,\cdot,\cdot,T)=0$ \emph{is} satisfied, we now turn to the classical work of \cite{berger,thorpe2} (even more important for us will be its Lorentzian counterpart \cite{thorpe}, in Section \ref{sec:L_counterpart} below); \cite{berger,thorpe2} showed that, at each $p\in M$, there is an orthonormal basis $\{e_1,e_2,e_3,e_4\} \subseteq T_pM$ relative to which $\hat{W}$ takes the form
\beqa
\label{eqn:WW1}
\cow = \begin{bmatrix}
\lambda_1 & 0 & 0 & \mu_1 & 0 & 0\\
0 & \lambda_2 & 0 & 0 & \mu_2 & 0\\
0 & 0 & \lambda_3 & 0 & 0 & \mu_3\\
\mu_1 & 0 & 0 & \lambda_1 & 0 & 0\\
0 & \mu_2 & 0 & 0 & \lambda_2 & 0\\
0 & 0 & \mu_3 & 0 & 0 & \lambda_3
\end{bmatrix},
\eeqa 
such that
\beqa
\label{eqn:Bid}
\underbrace{\,\lambda_1+\lambda_2+\lambda_3 = 0\,}_{\text{because $W$ is trace-free}} \comma \underbrace{\,\mu_1+\mu_2+\mu_3 = 0\,,}_{\text{because $W$ satisfies the algebraic Bianchi identity}}
\eeqa
and such that the 2-planes
$
P_1 \defeq \ww{e_1}{e_2}, P_2 \defeq \ww{e_1}{e_3}, P_3 \defeq \pm\ww{e_1}{e_4},
$ satisfy
$
\hat{W}P_i = \lambda_iP_i + \mu_i \hsr\!P_i, i=1,2,3,
$
with the $\lambda_i,\mu_i$ being completely determined by the critical points and values of the quadratic form \eqref{eqn:qaudW}; in fact the $\lambda_i$'s are (some of) its critical values. (E.g., the Weyl tensor of \eqref{eqn:*} is in normal form relative to the orthonormal basis $\{\partial_r/\sqrt{8x^3},\partial_x,\sqrt{8x^3}\partial_y,\partial_z\}$.) The normal form \eqref{eqn:WW1} readily yields a pointwise algebraic classification of oriented Riemannian 4-manifolds satisfying \eqref{eqn:main}, given by the following overdetermined system:

\begin{thm}
\label{thm:alg}
Let $(M,g)$ be an oriented Riemannian 4-manifold whose Weyl tensor $W$ has normal form \eqref{eqn:WW1} at $p\in M$. Then
$$
W(T,\cdot,\cdot,T)=0 \phantom{\comma} \text{for $T \defeq \sum_i c_ie_i \in T_pM$}
$$
if and only if the following ten equations hold for the $c_i$:
\beqa
\underbrace{\,c_2^2\lambda_1 +c_3^2\lambda_2+c_4^2\lambda_3 = 0\,}_{\text{$W(v,e_1,e_1,v)=0$}} &\comma& \underbrace{\,c_1c_2\lambda_1 +c_3c_4(\mu_2-\mu_3) = 0\,}_{\text{$W(v,e_1,e_2,v)=0$}},\nonumber\\
\underbrace{\,c_1c_3\lambda_2 +c_2c_4(-\mu_1+\mu_3) = 0\,}_{\text{$W(v,e_1,e_3,v)=0$}} &\comma& \underbrace{\,c_1c_4\lambda_3 +c_2c_3(\mu_1-\mu_2) = 0\,}_{\text{$W(v,e_1,e_4,v)=0$}},\nonumber\\
\underbrace{\,c_1^2\lambda_1 +c_4^2\lambda_2+c_3^2\lambda_3 = 0\,}_{\text{$W(v,e_2,e_2,v)=0$}} &\comma& \underbrace{\,c_2c_3\lambda_3 +c_1c_4(\mu_1-\mu_2) = 0\,}_{\text{$W(v,e_2,e_3,v)=0$}},\nonumber\\
\underbrace{\,c_2c_4\lambda_2 +c_1c_3(-\mu_1+\mu_3) = 0\,}_{\text{$W(v,e_2,e_4,v)=0$}} &\comma& \underbrace{\,c_4^2\lambda_1 +c_1^2\lambda_2+c_2^2\lambda_3 = 0\,}_{\text{$W(v,e_3,e_3,v)=0$}},\nonumber\\\underbrace{\,c_3c_4\lambda_1 +c_1c_2(\mu_2-\mu_3) = 0\,}_{\text{$W(v,e_3,e_4,v)=0$}} &\comma& \underbrace{\,c_3^2\lambda_1 +c_2^2\lambda_2+c_1^2\lambda_3 = 0\,}_{\text{$W(v,e_4,e_4,v)=0$}}.\nonumber
\eeqa
\end{thm}

\begin{proof}
All of these follow from expanding
$$
W(T,e_j,e_k,T) = \sum_{i,l}c_ic_lW_{ijkl}
$$
and using \eqref{eqn:WW} and \eqref{eqn:WW1}.
\end{proof}

More important for us will be the ``Lorentzian" version of these equations in Theorem \ref{thm:onlyI}, but let us use Theorem \ref{thm:alg} now to verify that nontrivial solutions do exist, and that, as we saw in \eqref{eqn:*}, they are not, in general, unique:

\begin{cor}
\label{cor:examples}
Assume that the Weyl tensor is nonzero in Theorem \ref{thm:alg}.  Then the vector $T = e_1+e_2+e_3+e_4$ satisfies $W(T,\cdot,\cdot,T)=0$ if and only if $\lambda_1 = -\mu_1-2\mu_2$ and $\lambda_2=2\mu_1+\mu_2$. If so, then so do the vectors $-e_1-e_2+e_3+e_4,-e_1+e_2-e_3+e_4$, and $-e_1+e_2+e_3-e_4$. On the other hand, if $\mu_2=\mu_3$ and $\lambda_2 = \lambda_3 \neq 0$, then the only solution is $T=0$.
\end{cor}

\begin{proof}
With $c_1=c_2=c_3=c_4=1$, and recalling \eqref{eqn:Bid}, the proofs of the first two statements follow easily from inspection of the equations above. Regarding the final statement, we must have $\lambda_1 = -2\lambda_2 \neq 0$; by the second equation, we have $c_1 = 0$ or $c_2 = 0$. On the other hand, the second-to-last equation yields $c_3 = 0$ or $c_4=0$. Any of these pairings, when inserted into the remaining equations, will yield $T= 0$. 
\end{proof}

In Corollary \ref{cor:examples}, if $\lambda_2 \neq \lambda_3$ then the last statement no longer holds, in that nontrivial solutions may exist; indeed, this was the case in \eqref{eqn:*}.

\section{Introducing a Lorentzian metric}
\label{sec:L_counterpart}
What Theorem \ref{thm:alg} does not address is what further \emph{geometry}, if any, lies behind $W(T,\cdot,\cdot,T) = 0$. As we now show, there is further geometry, if one is prepared to change the signature of $g$, to that of \emph{Lorentzian} metric. To see how, let us start with the observation that, in the basis \eqref{eqn:Hbasis0} with $T=e_1$, and recalling \eqref{eqn:WW},
\beqa
\label{eqn:cond2}
W(T,\cdot,\cdot,T) = 0 \iff \cow = \begin{bmatrix}
O & B\\
B & O
\end{bmatrix}\cdot
\eeqa

As we now show, \eqref{eqn:cond2} ensures that $W$ is also trace-free with respect to \emph{another} metric\,---\,and that as a consequence it will commute with that metric's Hodge star operator:

\begin{prop}
\label{prop:W1}
Let $(M,g)$ be an oriented Riemannian 4-manifold with Weyl tensor $W$ and let $T$ be a unit-length vector field on $M$. With respect to the Lorentzian metric \emph{$\gL \defeq g - 2T^{\flat} \otimes T^{\flat}$}, the linear endomorphism $\coww\colon\Lambda^2\lra \Lambda^2$ defined by
\beqa
\label{eqn:WR}
\text{\emph{$\ipl{\coww(v\wedge w)}{x\wedge y} \defeq W(v,w,x,y)$}}\hspace{.2in}\text{for all $v,w,x,y \in T_pM$}
\eeqa
commutes with the Hodge star \emph{$\hsl\colon \Lambda^2 \lra \Lambda^2$} of \emph{$\gL \iff W(T,\cdot,\cdot,T) = 0$}. If so, then $\coww$ is a complex-linear map on the three-dimensional complex vector space $\cx$ defined via \emph{$i\xi \defeq \hsl \xi$} for all $\xi \in \Lambda^2$. 
\end{prop}

Before beginning the proof, note that $\coww$ is merely the endomorphism of $W$ induced via $\ipl{\,}{}$ in place of $\ipr{\,}{}$; by construction, it is $\ipl{\,}{}$-self-adjoint and satisfies the algebraic Bianchi identity with respect to it. In terms of a $\gL$- and $g$-orthonormal basis $\{T= e_1,e_2,e_3,e_4\}$, $\coww$ takes the form
\beqa
\hspace{-.2in}\coww \!\!&=&\!\! \begin{bmatrix}
-W_{1212} & -W_{1312} & -W_{1412} & -W_{3412} & -W_{4212} & -W_{2312}\\
-W_{1213} & -W_{1313} & -W_{1413} & -W_{3413} & -W_{4213} & -W_{2313}\\
-W_{1214} & -W_{1314} & -W_{1414} & -W_{3414} & -W_{4214} & -W_{2314}\\
W_{1234} & W_{1334} & W_{1434} & W_{3434} & W_{4234} & W_{2334}\\
W_{1242} & W_{1342} & W_{1442} & W_{3442} & W_{4242} & W_{2342}\\
W_{1223} & W_{1323} & W_{1423} & W_{3423} & W_{4223} & W_{2323}
\end{bmatrix}\label{eqn:WW2*}\\
&=&\!\! \begin{bmatrix}
-A & -B\\
B & A
\end{bmatrix}\cdot\label{eqn:WW2}
\eeqa
What accounts for this difference with \eqref{eqn:WW} is that $\gL(e_1,e_1) = -1$, so that
$$
\ipl{\ww{e_1}{e_i}}{\ww{e_1}{e_i}} = -1 \comma i=2,3,4.
$$
Another difference is that $\gL$'s Hodge star operator $\hsl\colon \Lambda^2\lra \Lambda^2$ satisfies
\beqa
\arraycolsep=1.4pt\def\arraystretch{1.5}
\left\{\begin{array}{lr}
\hsl(\ww{e_1}{e_2}) = -\ww{e_3}{e_4},\\
\hsl(\ww{e_1}{e_3}) = -\ww{e_4}{e_2},\\
\hsl(\ww{e_1}{e_4}) = -\ww{e_2}{e_3},
\end{array}\right.  \comma \arraycolsep=1.4pt\def\arraystretch{1.5}
\left\{\begin{array}{lr}
\hsl(\ww{e_3}{e_4}) = \ww{e_1}{e_2},\\
\hsl(\ww{e_4}{e_2}) = \ww{e_1}{e_3},\\
\hsl(\ww{e_2}{e_3}) = \ww{e_1}{e_4}\label{eqn:minus}
\end{array}\right.
\eeqa
($\hsl$ is defined via $\ww{\xi}{\hsl\eta} \defeq \ipl{\xi}{\eta}\,dV$), or in block matrix form,
\beqa
\label{eqn:Hbasis}
\hsl = \begin{bmatrix}
O & I\\
-I & O
\end{bmatrix}\cdot
\eeqa

With these established, the proof of Proposition \ref{prop:W1} is now an easy matter:

\begin{proof}
Set $\{T\defeq e_1,e_2,e_3,e_4\}$. If $W(T,\cdot,\cdot,T) = 0$, then by \eqref{eqn:cond2} and \eqref{eqn:WW2},
$$
\hspace{-.2in}\coww = \begin{bmatrix}
O & -B\\
B & O
\end{bmatrix},
$$
which is precisely the condition needed to commute with \eqref{eqn:Hbasis}:
\beqa
\begin{bmatrix}
-A & -B\\
B & A
\end{bmatrix}\begin{bmatrix}
O & I\\
-I & O
\end{bmatrix} = \begin{bmatrix}
O & I\\
-I & O
\end{bmatrix}\begin{bmatrix}
-A & -B\\
B & A
\end{bmatrix} \iff A = O.\label{eqn:A=O}
\eeqa
The remainder of Proposition \ref{prop:W1} now follows from the fact that $\hsl^2 = -\text{id}$.
\end{proof}

Being $\gL$-trace-free, and commuting with $\hsl$, are equivalent for $\coww$\,---\,and the point of Proposition \ref{prop:W1} is that they are guaranteed by the condition $W(T,\cdot,\cdot,T) = 0$. It is also a concrete geometric realization of the general notion of ``$\hsl$-Einstein metric" formulated in \cite[Definition~3]{aazami}. In any case, with this Lorentzian structure in hand, we now turn to classifying these metrics using the technique of \cite{thorpe}. As in the latter, the key will be the eigenstructure of the (now) complex-linear map $\coww\colon \cx\lra \cx$:

\begin{defn}[Petrov Type]
\label{def:Ptype}
Let $(M,g)$ be an oriented Riemannian 4-manifold whose Weyl tensor $W$ satisfies $W(T,\cdot,\cdot,T) = 0$ for some unit-length vector field $T$ on $M$.  Then $(M,g)$ has \emph{Petrov Type I, D, II, N, or III} at each $p\in M$ if the complex-linear map \emph{$\coww\colon \cx \lra \cx$} has 3 \emph{(I}, \emph{D)}, 2 \emph{(II}, \emph{N)}, or 1 \emph{(III)} linearly independent complex eigenvectors at $p$, respectively, with \emph{I} and \emph{D} having 3 and 2 distinct complex eigenvalues, respectively, and \emph{II} and \emph{N} having 2 and 1 distinct complex eigenvalues, respectively.   
\end{defn}

(If $(M,g)$ has only 1 distinct complex eigenvalue, then that eigenvalue is necessarily 0 because $W$, hence $\coww$, is $\gL$-trace-free; note that we ignore the $W=0$ case in Definition \ref{def:Ptype}.) For the next definition, recall that an oriented 2-dimensional subspace $P \subseteq T_pM$ is \emph{nondegenerate} if the restriction of $\gL$ to $P$ is nondegenerate. The \emph{sign of $P$}, denoted $\vepo(P) = \pm1$, is defined to be $-1$ if this restriction has Lorentzian signature and $+1$ if it is positive-definite.  The 2-plane \emph{$\gL$-orthogonal} to $P$
 is $\hsl P$, since $\ipl{P}{\hsl P} = 0$. Finally, following \cite{thorpe}, let $G_{\pm}(p) \subseteq \Lambda^2(T_pM)$ denote the set of all oriented, nondegenerate 2-dimensional subspaces of $T_pM$ with signs $\pm1$, respectively.  Note that $\vepo(P) = \ipl{P}{P}$ for any $P \in G_{\pm}(p)$.

\begin{defn}[$\gL$-quadratic form of $W$]
\label{def:T_sec0coo1}
Assume the hypotheses of Definition \ref{def:Ptype}. Then the function \emph{$\Wsec$}, defined on each $G_+(p)\cup G_-(p) \subseteq T_pM$ by
\beqa
\label{def:T_seccoo1}
\text{\emph{$\Wsec$}}(P) \defeq \text{\emph{$\vepo(P)\ipl{\coww P}{P}$}},
\eeqa
with $\coww$ given by \eqref{eqn:WR}, is the \emph{$\gL$-quadratic form of the Weyl tensor $W$ of $g$}.
\end{defn}

Because $\coww$ is $\ipl{\,}{}$-self-adjoint, its Petrov Types are related to the critical points of $\Wsec$ \emph{by exactly the same method of proof as in} \cite{thorpe}:

\begin{thm}[Classification of Petrov Types]
\label{thm:n2}
Assume the hypotheses of Definition \ref{def:Ptype}.  Then at each $p \in M$, $g$ has
\begin{enumerate}[leftmargin=*]
\item[1.] Petrov Type I \emph{$\!\!\iff\!\! \Wsec$} has $n=3$ spacelike critical points.
\item[2.] Petrov Types D or N \emph{$\!\!\iff\!\! \Wsec$} has $n=\infty$ spacelike critical points.
\item[3.] Petrov Type II \emph{$\!\!\iff\!\! \Wsec$} has $n=1$ spacelike critical point.
\item[4.] Petrov Type III \emph{$\!\!\iff\!\! \Wsec$} has $n=0$ spacelike critical points.
\end{enumerate}
\end{thm}

\begin{proof}
Notice the slight difference from \cite{thorpe}; see the Appendix.
\end{proof}

There always exists a $\gL$-orthonormal basis $\{e_1,e_2,e_3,e_4\} \subseteq T_pM$\,---\,\emph{with $e_1$ timelike but not necessarily equal to $T$}\,---\,relative to which $\coww\colon\Lambda^2\lra\Lambda^2$ (i.e., as a real, not complex map) takes one of the following forms: 
\beqa
\label{eqn:Kerr}
\underbrace{\,\begin{bmatrix}
\lambda_1 & 0 & 0 & \mu_1 & 0 & 0\\
0 & \lambda_2 & 0 & 0 & \mu_2 & 0\\
0 & 0 & \lambda_3 & 0 & 0 & \mu_3\\
-\mu_1 & 0 & 0 & \lambda_1 & 0 & 0\\
0 & -\mu_2 & 0 & 0 & \lambda_2 & 0\\
0 & 0 & -\mu_3 & 0 & 0 & \lambda_3
\end{bmatrix},\,}_{\text{Petrov Types I or D (if D, then $\lambda_2=\lambda_3$ and $\mu_2=\mu_3$); $\sum_i\lambda_i = \sum_i \mu_i=0$}}
\eeqa
or
\beqa
\label{eqn:Kerr2}
\underbrace{\,\begin{bmatrix}
-2\lambda & 0 & 0 & -2\mu & 0 & 0\\
0 & \lambda-\frac{1}{2} & 0 & 0 & \mu & -\frac{1}{2}\\
0 & 0 & \lambda+\frac{1}{2} & 0 & -\frac{1}{2} & \mu\\
2\mu & 0 & 0 & -2\lambda & 0 & 0\\
0 & -\mu & \frac{1}{2} & 0 & \lambda-\frac{1}{2} & 0\\
0 & \frac{1}{2} & -\mu & 0 & 0 & \lambda+\frac{1}{2}
\end{bmatrix},\,}_{\text{Petrov Types II or N (if N, then $\lambda=\mu=0$)}}
\eeqa
or
\beqa
\label{eqn:Kerr3}
\underbrace{\,\frac{1}{\sqrt{2}}\!\begin{bmatrix}
0 & 1 & 0 & 0 & 0 & 0\\
1 & 0 & 0 & 0 & 0 & -1\\
0 & 0 & 0 & 0 & -1 & 0\\
0 & 0 & 0 & 0 & 1 & 0\\
0 & 0 & 1 & 1 & 0 & 0\\
0 & 1 & 0 & 0 & 0 & 0
\end{bmatrix}\,}_{\text{Petrov Type III}}\cdot
\eeqa
This follows by Jordan-normal form theory; cf.~\cite[pp.~3-4]{thorpe} (note that the timelike vector there is $e_4$, not $e_1$) and \cite[pp.~314~\&~324]{o1995}. However, in contrast to the Riemannian case, here the $\lambda_i,\mu_i$'s cannot be determined by just the first derivatives (i.e., critical points and values) of the quadratic form\,---\,the second derivatives of $\Wsec$ (i.e., its Hessian) are needed, too; see \cite[Footnote~10]{thorpe}. Nevertheless, we have at the moment two ``nice" bases: A $g$-orthonormal basis \`a la \cite{berger,thorpe2} giving the normal form \eqref{eqn:WW1} for $\cow$, and a $\gL$-orthonormal (Jordan) basis \`a la \cite{thorpe} giving one of the normal forms \eqref{eqn:Kerr}-\eqref{eqn:Kerr3} for $\coww$. Generally speaking, we would not expect the $g$-basis to contain $T$, hence we have no reason to expect that basis to be $\gL$-orthonormal also. At the same time, $e_1$ need not equal $T$ in the $\gL$-Jordan basis, either (if we could ensure that $e_1=T$, then we could immediately rule out Types II, N, and III, because $W(T,\cdot,\cdot,T) = 0$ would force $A=O$). Therefore, it is not clear at the moment what possible Petrov Types $\coww$ can actually have. We answer this question now:

\begin{thm}
\label{thm:onlyI}
Let $(M,g)$ be an oriented Riemannian 4-manifold whose Weyl tensor $W$ satisfies $W(T,\cdot,\cdot,T) = 0$ for some unit-length vector field $T$ on $M$. Then $(M,g)$ can only have Petrov Types I or D. Thus its Petrov Type is completely determined by whether \emph{$\Wsec$} has $3$ or $\infty$ critical points.
\end{thm}

\begin{proof}
As mentioned above, if the (trace-free) Weyl tensor $W$ has Petrov Type III at $p \in M$, then a $\gL$-orthonormal basis $\{e_1,e_2,e_3,e_4\} \subseteq T_pM$ can be found, with $e_1$ timelike, such that relative to the corresponding basis \eqref{eqn:Hbasis0} for $\Lambda^2$, the real $6 \times 6$ matrix of $\coww\colon\Lambda^2\lra\Lambda^2$ is given by \eqref{eqn:Kerr3}. As $\ipl{\,}{}$ has index $(-\!-\!-\!+\!++)$, the only nonzero components of $W$ in this basis are
\beqa
\label{eqn:III*}
W_{1213} = -1/\sqrt{2} \commas W_{1323} = W_{1442} = W_{3442} = 1/\sqrt{2}.
\eeqa
We may now repeat the procedure of Theorem \ref{thm:alg}; i.e., set $T = \sum_ic_ie_i$ and list the equations $\WL(T,e_i,e_j,T) = 0$ (though this time $g_{11} = g^{11} = -1$, as $e_1$ is timelike). But our task is made simpler by the fact that, via  \eqref{eqn:III*}, the two equations 
$$
\underbrace{\,c_3^2-c_1c_3-c_4^2 = 0\,}_{\text{$W(T,e_1,e_2,T)=0$}} \comma \underbrace{\,c_1^2-c_1c_3+c_4^2 = 0\,}_{\text{$W(T,e_2,e_3,T)=0$}}
$$
alone yield $(c_1-c_3)^2=0$, hence $c_1=c_3$. But with $e_1$ the timelike direction, such a $T$ cannot be timelike. Thus Petrov Type III is excluded.  For Petrov Type II, we have instead the matrix \eqref{eqn:Kerr2}. This time, the relevant starting equation is $W(T,e_3,e_4,T) = -2c_3c_4\lambda=0$. If $\lambda = 0$, then the equation $W(T,e_3,e_3,T)=0$ with $\lambda = 0$ yields $(c_1-c_2)^2 = 0$; but any $T$ with $c_1=c_2$ cannot be timelike. On the other hand, if $c_3 = 0$, then
$$
\underbrace{\,c_4(c_1+(2\lambda-1)c_2) = 0\,}_{\text{$W(T,e_2,e_4,T)=0$}} \commas \underbrace{\,(c_1-c_2)((2\lambda+1)c_1+(2\lambda-1)c_2) = 0\,}_{\text{$W(T,e_4,e_4,T)=0$}}.
$$
Since $c_1 \neq c_2$, the second equation yields $(2\lambda-1)c_2 = -(2\lambda+1)c_1$. Inserting this into the first gives $c_1c_4\lambda=0$, hence $c_4 = 0$. Finally, inserting this into $W(T,e_2,e_2,T) = 0$ yields $c_1^2\lambda = 0$, an impasse once again. Thus the cases $\lambda = 0$ and $c_3 = 0$ in $W(T,e_3,e_4,T) = -2c_3c_4\lambda=0$ have now been established; the $c_4 = 0$ case is similar to $c_3=0$ (now with $W(T,e_3,e_3,T)$ and $W(T,e_2,e_3,T)$ replacing $W(T,e_4,e_4,T)$ and $W(T,e_2,e_4,T)$).
\end{proof}

\section{The Lorentzian case}
\label{sec:Lor}
We now go in the opposite direction, starting from a Lorentzian 4-manifold $(M,\gL)$. Here we shall need $T$ to be timelike from the outset: $\gL(T,T) < 0$. Our first task will be to establish the analogue of Proposition \ref{thm:onlyI}. To that end, let us briefly define the \emph{Petrov Types} from general relativity.  Recall the complex-linear map $\hat{W}_{\scalebox{0.4}{$L$}}\colon\cx\lra\cx$, defined via $\ipl{\hat{W}_{\scalebox{0.4}{$L$}}(\ww{v}{w})}{\ww{x}{y}} \defeq \WL(v,w,x,y)$ and the identification $i\xi = \hsl\xi$ making $\Lambda^2$ intro a three-dimensional complex vector space $\cx$ (note that $\WL$ automatically commutes with $\hsl$, as the former is $\gL$-trace-free). An oriented Lorentzian 4-manifold $(M,\gL)$ with nonzero Weyl tensor has \emph{Petrov Type I, D, II, N, or III} at $p\in M$ if $\hat{W}_{\scalebox{0.4}{$L$}}\colon \cx \lra \cx$ has 3 (I or D), 2 (II or N), or 1 (III) linearly independent complex eigenvectors at $p$, with I and D having 3 and 2 distinct complex eigenvalues, respectively, and II and N having 2 and 1 distinct eigenvalues, respectively; see \cite[Lemma~5.4.2~p.~313]{o1995}. (Indeed, Definition \ref{def:Ptype} above is simply the mirror of this definition, with $(M,g,\coww)$ in place of $(M,\gL,\hat{W}_{\scalebox{0.4}{$L$}})$.) Petrov Type D includes the \emph{Kerr metric} modeling a rotating black hole, while Petrov Type N includes \emph{pp-wave spacetimes} modeling gravitational waves; see \cite{o1995,stephani,AMS}.

\begin{thm}
\label{prop:N}
If a Lorentzian 4-manifold with nonzero Weyl tensor $\WL$ satisfies $\WL(T,\cdot,\cdot,T) = 0$ for a timelike vector $T$, then it cannot have Petrov Types II, N, or III.
\end{thm}

\begin{proof}
This proof goes through exactly as in Theorem \ref{thm:onlyI}, with $\WL,\hat{W}_{\scalebox{0.4}{$L$}}$ in place of $W,\coww$, respectively.
\end{proof}

We now proceed to classifying those Lorentzian 4-manifolds that can satisfy $\WL(T,\cdot,\cdot,T) = 0$, by using $T$ to define a Riemannian metric $g \defeq \gL + 2T^{\flat} \otimes T^{\flat}$.  The first step toward this classification is the analogue of Proposition \ref{prop:W1}:

\begin{prop}
\label{prop:WT1}
Let $(M,\gL)$ be an oriented Lorentzian 4-manifold with Weyl tensor $\WL$ and let $T$ be a unit-length timelike vector field on $M$. With respect to the Riemannian metric \emph{$g \defeq \gL + 2T^{\flat} \otimes T^{\flat}$}, with $T^{\flat} \defeq \gL(T,\cdot)$, the linear endomorphism $\cowT\colon \Lambda^2 \lra \Lambda^2$ defined by
\beqa
\label{eqn:WL}
\text{\emph{$\ipr{\cowT(v\wedge w)}{x\wedge y} \defeq \WL(v,w,x,y)$}}\hspace{.2in}\text{for all $v,w,x,y \in T_pM$}
\eeqa
commutes with the Hodge star \emph{$\hsr\colon \Lambda^2 \lra \Lambda^2$} of $g \iff \WL(T,\cdot,\cdot,T) = 0$. 
\end{prop}

\begin{proof}
With respect to a $\gL$- and $g$-orthonormal basis $\{T\defeq e_1,e_2,e_3,e_4\}$, and setting $W_{ijkl} \defeq \WL(e_i,e_j,e_k,e_l)$, the operator $\cowT$ is given by 
$$
\cowT = \begin{bmatrix}
W_{1212} & W_{1312} & W_{1412} & W_{3412} & W_{4212} & W_{2312}\\
W_{1213} & W_{1313} & W_{1413} & W_{3413} & W_{4213} & W_{2313}\\
W_{1214} & W_{1314} & W_{1414} & W_{3414} & W_{4214} & W_{2314}\\
W_{1234} & W_{1334} & W_{1434} & W_{3434} & W_{4234} & W_{2334}\\
W_{1242} & W_{1342} & W_{1442} & W_{3442} & W_{4242} & W_{2342}\\
W_{1223} & W_{1323} & W_{1423} & W_{3423} & W_{4223} & W_{2323}
\end{bmatrix}= \begin{bmatrix}
A & B\\
B & -A
\end{bmatrix}\cdot\nonumber
$$
(Entry-by-entry, $\hat{W}_{\scalebox{0.4}{$L$}}$ has the same appearance as \eqref{eqn:WW2*}, though with $\WL$ in place of $W$; since $\hat{W}_{\scalebox{0.4}{$L$}}$ (always) commutes with $\hsl$, its block form is $\hat{W}_{\scalebox{0.4}{$L$}} =$ {\tiny $\begin{bmatrix}
-A & -B\\
B & -A
\end{bmatrix}$}. As a consequence, $\cow$ takes the block form $\cowT =$ {\tiny $\begin{bmatrix}
A & B\\
B & -A
\end{bmatrix}$}.) Once again, $A=O$ when $\WL(T,\cdot,\cdot,T) = 0$, and this occurs if and only if $\cowT$ commutes with the Hodge star $\hsr$ of $g$ (recall \eqref{eqn:HbasisR}).
\end{proof}

\begin{defn}[$g$-quadratic form of $\WL$]
\label{def:RiemT}
Let $(M,g)$ be an oriented Lorentzian 4-manifold whose Weyl tensor $\WL$ satisfies $W(T,\cdot,\cdot,T) = 0$ for some unit-length timelike vector field $T$ on $M$.  Then the function \emph{$\WTsec$}, defined on each 2-plane $P \subseteq T_pM$ by
\beqa
\label{def:RT1}
\text{\emph{$\WTsec$}}(P) \defeq \text{\emph{$\ipr{\cowT P}{P}$}},
\eeqa
with $\cowT$ given by \eqref{eqn:WL}, is the \emph{$g$-quadratic form of the Weyl tensor $\WL$ of $\gL$}.
\end{defn}

We now arrive at the analogue of Theorem \ref{thm:n2}:

\begin{thm}
\label{thm:n2R}
Let $(M,g)$ be an oriented Lorentzian 4-manifold whose Weyl tensor $\WL$ satisfies $W(T,\cdot,\cdot,T) = 0$ for some unit-length timelike vector field $T$ on $M$. Then at each $p\in M$, $\WL$ is completely determined by the critical points and values of \emph{$\WTsec$}.
\end{thm}

\begin{proof}
The proof is identical to \cite[Theorems~2.1,~2.2]{thorpe2} (recall \eqref{eqn:WW1} and \eqref{eqn:Bid}), the only difference being that $\cowT$ is not the curvature operator $\co$ of $g$ used in \cite{thorpe2}. (The proof in \cite{thorpe2} still goes through here because $\cowT$ is $\ipr{\,}{}$-self-adjoint and commutes with $\hsr$.)
\end{proof}

\section*{Appendix: Details of the proof of Theorem \ref{thm:n2}}
The key ingredients in the proof of Theorem \ref{thm:n2} are the following canonical bases associated to each Petrov Type:
\begin{prop}
\label{thm:JordanPcoo}
Let $\coww\colon\cx\lra\cx$ be the complex-linear map of Definition \ref{def:Ptype}. Then at each $p \in M$, the following is true\emph{:}
\begin{enumerate}[leftmargin=*]
\item[1.] For Petrov Types I or D, \emph{$\coww$} has a basis of orthogonal spacelike 2-planes.
\item[2.] For Petrov Types II or N, \emph{$\coww$} has one spacelike and one lightlike 2-plane as eigenvectors, orthogonal to each other.
\item[3.] For Petrov Type III, the one linearly independent eigenvector of \emph{$\coww$} is necessarily a lightlike 2-plane.
\end{enumerate}
\end{prop}

\begin{proof}
(This proof is identical \cite[Theorem,~p.~3]{thorpe} (see also \cite[Remark~(ii),~p.~4]{thorpe}), though we write it out in detail here since our operator $\coww$ is not the Lorentzian curvature operator to be found therein, since our presentation differs from \cite{thorpe} in that we omit the derivation of the Jordan forms \eqref{eqn:Kerr}, and finally, since we have separated Types D and N from I and II, respectively, for the sake of Theorem \ref{thm:n2} below.)  The key to finding the bases alleged is to work with the following complex scalar product on the three-dimensional complex vector space $\cx$:
$$
\ggL(\xi,\eta) \defeq \ipl{\xi}{\eta} - i\ipl{\xi}{\hsl \eta}
$$
Note that $\coww$ is $\ggL$-self-adjoint, because it is $\ipl{\,}{}$-self-adjoint and commutes with $\hsl$. Note also that $\xi \in \Lambda^2$ corresponds to a nondegenerate 2-plane in $T_pM$ if and only if $\ggL(\xi,\xi) = \pm1$, since
\beqa
\label{eqn:Pbasis}
\ggL(\xi,\xi) = \pm1 \iff \ipl{\xi}{\xi} = \pm1~~\text{and}~~\ipl{\xi}{\hsl\,\xi} = 0 \iff \xi \in G_\pm(p).
\eeqa
(For proofs of these facts, consult \cite[pp.~306 \& 362]{o1995}.)
\vskip 6pt
\underline{Petrov Type I or D at $p$}:~Let $\{\xi_1,\xi_2,\xi_3\} \subseteq \cx$ be a basis of eigenvectors for $\coww$, with corresponding eigenvalues $\lambda_1,\lambda_2,\lambda_3 \in \mathbb{C}$.  If $\lambda_i  \neq \lambda_j$, then $\ggL(\xi_i,\xi_j) =  0$ because $\coww$ is $\ggL$-self-adjoint. This immediately implies that if $\lambda_1,\lambda_2,\lambda_3$ are distinct and some $\ggL(\xi_j,\xi_j) = 0$, then that $\xi_j$ will be $\ggL$-orthogonal to \emph{all} 2-vectors in $\cx$\,---\,impossible as $\ggL$ is nondegenerate.  Thus each $\ggL(\xi_j,\xi_j) \neq 0$ if $\lambda_1,\lambda_2,\lambda_3$ are distinct, If the latter is true, then we may choose $\alpha_1,\alpha_2,\alpha_3 \in  \mathbb{C}$ so that $\ggL(\alpha_j\xi_j,\alpha_j\xi_j) = +1$; i.e., so that each eigenvector is a spacelike 2-plane, by \eqref{eqn:Pbasis}.  (Such  complex scalar multiplication will neither change their status as eigenvectors of $\coww$ nor change their eigenvalues.) This takes care of the case when $\lambda_1,\lambda_2,\lambda_3$ are distinct.  Suppose now that two of them are equal, say $\lambda_1 = \lambda_2 \neq \lambda_3$ (as we are still within the case of Petrov Type I, we still have a basis $\{\xi_1,\xi_2,\xi_3\}$ of eigenvectors).  Then $\ggL(\xi_3,\xi_3) \neq 0$, by the same reasoning as above.  Consider now the two-dimensional $\lambda_1$-eigenspace; it, too, must be $\ggL$-nondegenerate, hence must contain a $\ggL$-orthogonal basis.  We thus have a $\ggL$-orthogonal basis $\{\xi_1,\xi_2,\xi_3\}$ of eigenvectors of $\coww$, no element of which is $\ggL$-lightlike; hence, as before, each can be scaled to be a spacelike 2-plane.  Finally, the case when all three eigenvalues are equal, $\lambda_1=\lambda_2=\lambda_3$, is trivial because then every 2-vector is an eigenvector of $\coww$.
\vskip 6pt
\underline{Petrov Type II or N at $p$}:~Given the two linearly independent eigenvectors $\{\xi_1,\xi_2\} \subseteq \cx$ with, say, $\xi_1$ having the larger algebraic multiplicity, there exists a 2-vector $\eta$ satisfying $(\coww - \lambda_1 I)\eta = \xi_1$. Then
\beqa
\label{eqn:a_null}
\ggL(\xi_1,\xi_1) = \ggL((\coww-\lambda_1 I)\eta,\xi_1) = \ggL(\eta,(\coww-\lambda_1 I)\xi_1) = 0,
\eeqa
so that $\xi_1$ must be $\ggL$-lightlike.  But this is the case if and only if 
$\ipl{\xi_1}{\xi_1} = \ipl{\xi_1}{\hsl\,\xi_1} = 0$, which is the case if and only if $\xi_1$ is a lightlike 2-plane in $T_pM$; i.e., of the form $\xi_1 = \ww{v}{w}$ and $\gL$-degenerate.  Consider now $\xi_2$.  We claim that it must be $\ggL$-orthogonal to $\xi_1$.  Indeed, if $\lambda_1 \neq \lambda_2$, then this follows from the $\ggL$-self-adjointness of $\coww$; if $\lambda_1 = \lambda_2$, then
$$
\ggL(\xi_1,\xi_2) = \ggL((\coww-\lambda_1 I)\eta,\xi_2) = \ggL(\eta,(\coww-\lambda_1 I)\xi_2) = 0.
$$
Now suppose $\ggL(\xi_2,\xi_2) = 0$ also, so that $\xi_1,\xi_2$ are $\ggL$-orthogonal and lightlike.  Because $\ggL$ is nondegenerate and $\{\xi_1,\xi_2,\eta\}$ is a basis, each $\ggL(\xi_j,\eta) \neq 0$. 
And yet, because any 2-vector $\beta \defeq \alpha_1\xi_1+\alpha_2\xi_2$ will be $\ggL$-lightlike and orthogonal to $\xi_1, \xi_2$, we may choose $\alpha_1,\alpha_2 \in \mathbb{C}$ so that $\ggL(\beta,\eta) = 0$, contradicting the nondegeneracy of $\ggL$.  Therefore we must have $\ggL(\xi_2,\xi_2) \neq 0$, in which case we may scale it to satisfy  $\ggL(\xi_2,\xi_2) = +1$.
\vskip 6pt
\underline{Petrov Type III at $p$}:~Let $\xi_1 \in \cx$ denote $\coww$'s lone linearly independent eigenvector, with eigenvalue $\lambda_1$.  Once again, there is a 2-vector $\eta$ satisfying $(\coww - \lambda_1 I)\eta = \xi_1$, so that, via \eqref{eqn:a_null}, $\xi_1$ is necessarily a lightlike 2-plane.
\end{proof}

To connect these bases with $\Wsec$, we will need to characterize its critical points in terms of $\coww$:

\begin{prop}
\label{prop:crit2}
Let \emph{$\Wsec$} be given by \eqref{def:T_seccoo1} of Definition \ref{def:T_sec0coo1}. Then at any $p \in M$, a 2-plane $P \in G_\pm(p)$ is a critical point of \emph{$\Wsec$} if and only if 
$$
\emph{$\coww$} P = aP + b (\hsl P)
$$
for some $a,b  \in \RR$.
\end{prop}

\begin{proof}
This proof follows exactly as in \cite[Lemma,~p.~5]{thorpe}, except that we are working with $\coww$ in place of the Lorentzian curvature operator; see also \cite[Proposition~1]{aazami} for a more general case.
\end{proof}

In any case, we now have the ingredients needed to prove Theorem \ref{thm:n2}:

\begin{theorem-non}[Classification of Petrov Types]
Assume the hypotheses of Definition \ref{def:Ptype}.  Then at each $p \in M$, $g$ has
\begin{enumerate}[leftmargin=*]
\item[1.] Petrov Type I \emph{$\!\!\iff\!\! \Wsec$} has $n=3$ spacelike critical points.
\item[2.] Petrov Types D or N \emph{$\!\!\iff\!\! \Wsec$} has $n=\infty$ spacelike critical points.
\item[3.] Petrov Type II \emph{$\!\!\iff\!\! \Wsec$} has $n=1$ spacelike critical point.
\item[4.] Petrov Type III \emph{$\!\!\iff\!\! \Wsec$} has $n=0$ spacelike critical points.
\end{enumerate}
\end{theorem-non}

\begin{proof}
(Our proof differs from \cite[Theorem,~p.~5]{thorpe} in step 2, as we discuss below; also, note once again that we are working with $\coww$, not the Lorentzian curvature operator as in \cite{thorpe}.)  By Proposition \ref{prop:W1}, $\coww$ is a complex-linear map on $\cx$. By Proposition \ref{prop:crit2}, any $P \in G_{\pm}(p)$ will be a critical point of $\Wsec$ if and only if it is a complex eigenvector of $\coww$:
$$
\coww P = (a+ib)P = aP + b \hsl\!P.
$$
In particular, if $\Wsec$ has a spacelike critical point $P$\,---\,i.e., $P \in G_+(p)$\,---\,then it has a timelike critical point, namely, $iP = \hsl P \in G_-(p)$, since the latter is, of course, an eigenvector of $\coww$ if $P$ was. Thus if $\Wsec$ has, say, no spacelike critical points (hence no timelike critical points, either), then no nondegenerate 2-planes can be eigenvectors of $\coww$; i.e., $\coww$ can only have $\ggL$-lightlike eigenvectors.  By Proposition \ref{thm:JordanPcoo}, only Petrov Type III satisfies this criterion. Likewise, if $\Wsec$ has precisely one spacelike critical point (hence precisely one timelike critical point), then $\coww$ can only have one eigenvector from $G_{+}(p) \subseteq \cx$.  By Proposition \ref{thm:JordanPcoo}, only Petrov Type II satisfies this criterion.  Similarly, if $\Wsec$ has three spacelike critical points, then $\coww$ must have three spacelike 2-planes as eigenvectors.  By Proposition \ref{thm:JordanPcoo}, only Petrov Type I satisfies this criterion.  Finally, suppose that two eigenvalues of $\coww$ are equal, $\lambda_2=\lambda_3 \defeq \lambda$, hence Types D or N. In either case, denote the two $\ggL$-orthogonal 2-plane eigenvectors, guaranteed by Proposition \ref{thm:JordanPcoo}, by $\ww{x}{y},\ww{v}{z}$. If Type D, then they are spacelike, and any linear combination $\xi \defeq a(\ww{x}{y})+b(\ww{v}{z})$ of them with $a^2+b^2=1$ will yield another spacelike 2-plane: $\ggL(\xi,\xi) = +1$ (recall \eqref{eqn:Pbasis}). Hence there will be infinitely many complex eigenvectors of $\coww$ in $G_+(p)$, hence infinitely many critical points of $\Wsec$, by Proposition \ref{prop:crit2}. But, as seems not to have been noted by \cite{thorpe}, the same is true for Type N: Although one of these 2-plane eigenvectors, say, $\ww{v}{z}$, is now $\ggL$-lightlike, nevertheless $\ggL(\xi,\xi) = +1$ for all $\xi \defeq \ww{x}{y}+b(\ww{v}{z})$ with $b \in \RR$. Thus, the case of infinitely many critical points determines the Petrov Type only up to D or N. (This would therefore also be the case for the curvature operator of an Einstein metric $\gL$ in the proof of \cite{thorpe}.)
\end{proof}

\section*{Acknowledgments}
This paper is dedicated to John A. Thorpe (1936-2021), whose two beautiful papers \cite{thorpe,thorpe2} have been a source of inspiration for the author.

\section*{References}
\renewcommand*{\bibfont}{\footnotesize}
\printbibliography[heading=none]
\end{document}